\documentclass[11pt]{article}
\usepackage{epigamath}

\usepackage[notext]{kpfonts}
\usepackage{baskervald}

\setpapertype{A4}

\usepackage[english]{babel}

\usepackage[dvips]{graphicx}     

\title{\vspace{-0.5cm}A characterization
of finite vector bundles on Gauduchon astheno-K\"ahler manifolds}
\titlemark{Finite vector bundles on Gauduchon astheno-K\"ahler manifolds}
\author{\vspace{0cm} Indranil Biswas and Vamsi Pritham Pingali}
\authoraddresses{
\authordata{Indranil Biswas}{\firstname{Indranil} \lastname{Biswas}\\
\institution{School of Mathematics,
Tata Institute of fundamental research, Homi Bhabha road, Mumbai 400005, India}\\
\email{indranil@math.tifr.res.in}}\\
\authordata{Vamsi Pritham Pingali}{\firstname{Vamsi Pritham}
\lastname{Pingali}\\
\institution{Department of Mathematics, Indian Institute of Science, Bangalore 560012, India}\\
\email{vamsipingali@iisc.ac.in}}
}
\authormark{I. Biswas and V. P. Pingali}
\date{\vspace{-5ex}} 
\journal{\'Epijournal de G\'eom\'etrie Alg\'ebrique} 
\acceptation{Received by the Editors on January 16, 2018, and in final form
on June 22, 2018. \\ Accepted on August 16, 2018.}

\newcommand{\articlereference}{Volume 2 (2018), Article Nr. 6}

 \usepackage[all]{xy}

\allowdisplaybreaks

\numberwithin{equation}{numsection}

\newtheorem{theorem}{Theorem}[section]

\newtheorem{lemma}[theorem]{Lemma}

\newtheorem{remark}[theorem]{Remark}

\newcommand{\pbp}{\mathcal{\partial \overline{\partial}}}

\begin{document}


\maketitle



\begin{prelims}


\def\abstractname{Abstract}
\abstract{A vector bundle on a projective variety $X$ is called finite if it satisfies a 
nontrivial polynomial equation with integral coefficients. Nori proved that a vector 
bundle $E$ on $X$ is finite if and only if there is a finite \'etale Galois covering $q\, :\,
\widetilde{X}\, \longrightarrow\, X$ and a $\text{Gal}(q)$--module
$V$, such that $E$ is isomorphic to the quotient of $\widetilde{X}\times V$ by the twisted
diagonal action of $\text{Gal}(q)$ \cite{No1}, \cite{No2}. Therefore, $E$ is
finite if and only if the pullback of $E$ to some finite \'etale 
Galois covering of $X$ is trivial. We prove the same statement when
$X$ is a compact complex manifold admitting a Gauduchon astheno-K\"ahler metric.}



\vspace{0.15cm}

\languagesection{Fran\c{c}ais}{%

\textbf{Titre. Une caract\'erisation des fibr\'es vectoriels finis sur les vari\'et\'es asth\'eno-k\"ahl\'eriennes Gauduchon.} Un fibr\'e vectoriel sur une vari\'et\'e projective $X$ est dit fini s'il satisfait une \'equation polynomiale non triviale \`a coefficients entiers. Nori a montr\'e qu'un fibr\'e vectoriel $E$ sur $X$ est fini si et seulement s'il existe un rev\^etement galoisien \'etale fini $q\, :\,
\widetilde{X}\, \longrightarrow\, X$ et un $\text{Gal}(q)$--module
$V$ tels que $E$ soit isomorphe au quotient de $\widetilde{X}\times V$ par l'action diagonale tordue de $\text{Gal}(q)$ \cite{No1}, \cite{No2}. Par cons\'equent, $E$ est
fini si et seulement si l'image r\'eciproque de $E$ sur un certain rev\^etement galoisien fini de $X$ est triviale. Nous prouvons le m\^eme \'enonc\'e lorsque 
$X$ est une vari\'et\'e complexe compacte admettant une m\'etrique de Gauduchon asth\'eno-k\"ahl\'erienne .}

\end{prelims}


\newpage

\setcounter{tocdepth}{1} \tableofcontents

\section{Introduction}\label{Introsec}

For a polynomial $f(x)\,=\, \displaystyle \sum_{i=0}^n a_ix^i$, where $a_i$ are nonnegative integers,
and a holomorphic vector bundle $E$ on a complex projective variety or more
generally on a compact complex manifold $X$, the vector bundle
$$
\bigoplus_{i=0}^n (E^{\otimes i})^{\oplus a_i}
$$
on $X$ is denoted by $f(E)$ (the vector bundle $E^0$ is defined to be the trivial line 
bundle ${\mathcal O}_X$). An algebraic (respectively, holomorphic) vector bundle $E$ on a projective
variety (respectively, compact complex manifold) $X$ is called 
\textit{finite} if there are two distinct polynomials $f$, $g$ with nonnegative 
integral coefficients, such that the vector bundle $f(E)$ is isomorphic to $g(E)$. It
is known that a vector bundle $E$ is finite if and only if
there is a finite collection of
vector bundles $\{F_j\}_{j=1}^m$ and nonnegative integers $\{a_{i,j}\}_{j=1}^m$
such that
$$E^{\otimes i}\,=\, \bigoplus_{j=1}^m (F_j)^{\oplus a_{i,j}}$$
for all $i\, \geq\, 1$ \cite[p.~35, Lemma~3.1]{No1}, \cite[p.~80, Lemma~3.1]{No2}. When $X$ 
is a compact Riemann surface, it was noted by Weil that $E$ is finite if it
admits a flat connection with finite monodromy that is compatible with 
holomorphic structure of $E$ \cite{We}. Note that the condition that $E$ admits a flat 
connection with finite monodromy is equivalent to the condition that the pullback of 
$E$ to some finite \'etale Galois covering of $X$ is holomorphically trivial.

In a foundational work, Nori proved that a vector bundle $E$ on a complex projective variety $X$ is finite if and only if
there is a finite \'etale Galois covering $q\, :\, \widetilde{X}\, \longrightarrow\, X$ and a $\text{Gal}(q)$--module
$V$, such that $E$ is isomorphic to the quotient of $\widetilde{X}\times V$ by the twisted diagonal action of $\text{Gal}(q)$
\cite{No1}, \cite{No2}. Therefore, the monodromy homomorphism of the flat connection on $E$, given by the trivial connection on
$\widetilde{X}\times V\,\longrightarrow \,\widetilde{X}$, coincides with the one given by the action on $V$ of
$\text{Gal}(q)$ (it is a quotient of the fundamental group of $X$).

Let $X$ be a compact complex manifold equipped with a Hermitian form $\omega$.
This pair $(X,\omega)$ is called Gauduchon if $\pbp \omega^{n-1}\,=\,0$, and
it is called astheno-K\"ahler if $\pbp \omega^{n-2}\,=\,0$ \cite{JY}. Any compact complex
manifold admits a Gauduchon metric \cite{Ga}, and any complex Hermitian surface is astheno-K\"ahler. 

Our aim here is to prove the following:

\begin{theorem}
Suppose $X$ is a compact complex manifold admitting a Hermitian metric
that is both Gauduchon and astheno-K\"ahler. Then a holomorphic vector
bundle $E$ over $X$ is finite if 
and only if it corresponds to a representation of a finite quotient of the 
fundamental group of $X$, or equivalently, if and only if $E$ admits a flat 
connection, compatible with its holomorphic structure, that has finite monodromy group.
\label{mainthm}
\end{theorem}

Note that a flat connection with finite monodromy is a unitary connection for a suitable
Hermitian structure.

Theorem \ref{mainthm} was proved earlier in \cite{scholbis} under the assumption that 
$X$ is a compact K\"ahler manifold. The proof in \cite{scholbis} crucially used a 
well-known theorem of Corlette--Simpson \cite{Si}. However, that theorem is available 
only in the K\"ahler context. Here we have been able to avoid using the 
Corlette--Simpson theorem.

For higher dimensional examples of Gauduchon astheno-K\"ahler manifolds,
see \cite{LY}, \cite{FGV}, \cite{LaUg}, \cite{LYZ} for instance.

\section{Vanishing of Chern classes}\label{vanishing}

In whatever follows, for a compact complex manifold $X$ we use the Bott-Chern cohomology
defined as $$H^{p,q}_{BC} 
(X)\,=\,\frac{ker(\partial : A^{p,q}(X) \rightarrow A^{p+1,q}(X))\cap ker(\bar{\partial} : A^{p,q}(X) \rightarrow A^{p,q+1}(X))}{Im(\pbp(A^{p-1,q-1}(X)))}\, $$
where $A^{p,q}(X)$ consists of smooth $(p,q)$-forms on $X$. These groups are finite 
dimensional. They coincide with the Dolbeault cohomology
groups for K\"ahler manifolds, as a consequence of the $\pbp$ lemma. If $h$ is a 
Hermitian metric on a holomorphic vector bundle $E$ over $X$, and $\Theta_h$ is the
curvature of its 
Chern connection, then for any characteristic class $\Phi$, its Chern-Weil form 
$\Phi(\Theta_h)$ is closed, and $\Phi(\Theta_{h_2})-\Phi(\Theta_{h_1}) \,=\, \pbp 
BC(h_1,h_2)$, where $BC$ is the Bott-Chern form. Therefore, we can think of $\Phi$ as 
an element of the Bott-Chern cohomology (with representatives furnished by using 
Chern connections). Moreover, if $E\,=\,S\bigoplus V$, then 
$[ch(E)]_{BC}\,=\,[ch(S)]_{BC}+[ch(V)]_{BC}$ (proven by taking a direct sum metric). 
Likewise, we have $[ch(E\otimes F)]_{BC}\,=\,[ch(E)]_{BC}[ch(F)]_{BC}$.

{}From now onwards, we denote a vector bundle $\bigoplus_{j=1}^m (F_j)^{\oplus a_{ij}}$ 
as $\displaystyle \sum_{j=1}^m a_{ij}F_j$.

\begin{theorem}
Suppose $E$ is a finite holomorphic vector bundle of rank $r$ on a compact complex manifold
$X$, i.e., there exists integers $m\geq 1$, $a_{i,j} \,\geq\, 0$, $r_j \,\geq\, 1$ with
$1\leq\, j\, \leq\, m$, $i\,\in\, {\mathbb N}^{>0}$,
and holomorphic vector bundles $V_j$ of ranks $r_j$ such that $E^{\otimes i}
\,=\,\displaystyle \sum_{j=1}^{m} a_{i,j} V_j$. Then $[\mathrm{ch}(E)]_{BC}\,=\,r$. 
\label{vanishingofchernclassesthm}
\end{theorem}

\begin{proof}
We induct on $i$ to show that $[ch_i(E)]_{BC}\,=\,0$ for all $i\, \geq\, 1$.

For $i\,=\,1$, note that $\frac{n}{r} [c_1(E)]_{BC}\,= \,\frac{[ch_1 (E^{\otimes n})]_{BC}}{r^n}$. On
the other hand, $E^{\otimes n} \,=\, \bigoplus_{j=1}^m a_{n,j} V_j$.
So, $$\frac{[ch_1(E^{\otimes n})]_{BC}}{r^n}
\,=\, \displaystyle\sum_{j=1}^m a_{n,j} \frac{[c_1(V_j)]_{BC}}{r^n}\,=\,
\displaystyle\sum_{j=1}^m r_j a_{n,j} \frac{[c_1(V_j)]_{BC}/r_j}{r^n}
\, .$$
Since $r^n \,=\, \displaystyle \sum_{j=1}^k r_ja_{n,j}$, it follows that
$\displaystyle \sum_{j=1}^m a_{n,j}\frac{[c_1(V_j)]_{BC}}{r^n}$ lies in a bounded subset of
$H^{1,1}_{BC}(X, \mathbb{C})$; indeed it lies in the convex subset generated by
$\{[c_1(V_j)/r_j]_{BC}\}_{j=1}^k$. However, $\frac{n}{r} [c_1(E)]_{BC}$ is unbounded as $n$
varies if $[c_1(E)]_{BC}$ is nonzero. Therefore, it follows that $[c_1(E)]_{BC} \,=\, 0$.

Suppose $[ch_i(E)]_{BC}\,=\,0$ for all $1\,\leq\, i\,\leq\, k-1$. Then,
$$\frac{[ch_k(E^{\otimes n})]_{BC}}{r^n}\,=\, \frac{n}{r} [ch_k(E)]_{BC}\, ,$$ and hence
$\frac{[ch_k(E^{\otimes n})]_{BC}}{r^n}$ is unbounded as $n$ varies if $[ch_k(E)]_{BC}$ is nonzero.
On the other hand, we have
$$
\displaystyle\sum_{j=1}^m a_{n,j} \frac{[ch_k(V_j)]_{BC}}{r^n}\,=\,
\displaystyle\sum_{j=1}^m r_j a_{n,j} \frac{[ch_k(V_j)]_{BC}/r_j}{r^n}\, ,
$$
so $\displaystyle\sum_{j=1}^m a_{n,j} \frac{[ch_k(V_j)]_{BC}}{r^n}$ lies in the convex subset
of $H^{1,1}_{BC}(X, \mathbb{C})$ generated by $\{[ch_k(V_j)/r_j]_{BC}\}_{j=1}^k$ and is hence bounded. Thus, as before, $[ch_k(E)]_{BC}=0$.
\hfill $\Box$
\end{proof}

Take a finite vector bundle $E$ and vector bundles $V_j$, $1\leq\, j\, \leq\, m$, as in Theorem 
\ref{vanishingofchernclassesthm}. We may assume that
\begin{enumerate}
\item[\rm (1)] each $V_j$ is indecomposable (if $V_j\,=\, \bigoplus_{k=1}^\ell W_k$ with
each $W_k$ indecomposable, then replace $V_j$ by $\{W_k\}_{k=1}^\ell$),

\item[\rm (2)] the isomorphism classes of $V_j$ are distinct, and

\item[\rm (3)] for each $1\leq\, j\, \leq\, m$, there is some $j'\, \geq\, 1$ such that
$a_{j',j}$ in Theorem \ref{vanishingofchernclassesthm} is nonzero (if $a_{i,j}\,=\, 0$ for all $i$
then discard $V_j$).
\end{enumerate}

\begin{lemma}\label{lem0}
Each of the above vector bundles $V_j$, $1\leq\, j\, \leq\, m$, is finite.
\end{lemma}

\begin{proof}
Take any $V_j$. Take any $k\, \geq\, 1$ such that $a_{k,j}\,\geq\, 1$. Then for every $n\, \geq\, 1$,
the vector bundle $V^{\otimes n}_j$ is a direct summand of
\begin{equation}\label{dsf}
E^{\otimes kn}\,=\, \displaystyle \sum_{\ell=1}^{m} a_{kn,\ell} V_\ell\, .
\end{equation}
If $W$ be a holomorphic vector bundle on $X$, and
$$
\displaystyle \sum_{\alpha=1}^{a} A_\alpha \,=\, W\,=\, \displaystyle \sum_{\alpha=1}^{b} B_\alpha\, ,
$$
where all $A_\alpha$ and $B_\alpha$ are indecomposable holomorphic vector bundles, then a
theorem of Atiyah says that
\begin{itemize}
\item $a\,=\, b$, and

\item there is a permutation $\sigma$ of $\{1,\, \cdots,\, a\}$ such that the
holomorphic vector bundle $A_\alpha$ is
holomorphically isomorphic to $B_{\sigma(\alpha)}$ for all $\alpha\, \in\, \{1,\, \cdots,\, a\}$.
\end{itemize}
(See \cite[p.~315, Theorem~3]{At}.) In view of this theorem, it follows from \eqref{dsf}
that the direct summand $V^{\otimes n}_j$ is a direct sum of
copies of $V_\ell$, $1\leq\, \ell\, \leq\, m$. Therefore, $V_i$ is finite.
\hfill $\Box$
\end{proof}

\section{A characterization of numerically flat bundles on Gauduchon astheno-K\"ahler 
manifolds}\label{DPSresult}

We recall the definition of a \emph{nef} vector bundle as in \cite{DPS}. A line bundle $L$ over 
a compact complex manifold $M$, equipped with a Hermitian metric $\omega$, is called nef if 
given any real number $\epsilon \,>\, 0$, there is a smooth Hermitian metric $h$ on $L$ such that 
$\Theta(h_{\epsilon})+\epsilon \omega$ is a nonnegative Hermitian form, where 
$\Theta(h_{\epsilon})$ is the Chern curvature form. This definition actually does not depend on 
the choice of $\omega$. A holomorphic vector bundle $V$ on $M$ is called nef if the 
tautological line bundle $\mathcal{O}_{\mathbb{P}(V)}(1)$ on $\mathbb{P}(V)$ is nef. A 
holomorphic vector bundle $V$ is called \emph{numerically flat} if both $V$ and $V^*$ are nef.

We note that Lemma 2.4 of \cite{scholbis}, which was obtained as a corollary of Theorem 1.12
of \cite{DPS}, is stated in the K\"ahler set-up, because \cite{scholbis} is entirely
in the K\"ahler set-up. However, Theorem 1.12
of \cite{DPS} is proved for compact complex manifolds. Hence Lemma 2.4 of \cite{scholbis}
remains valid for compact complex manifolds.

\begin{lemma}[{\cite[Lemma 2.4]{scholbis}}]
If $E$ is finite, then $E$ is numerically flat.
\label{nefness}
\end{lemma}

\begin{proof}
Since $E$ is finite, there are finitely many indecomposable bundles $V_1,\,\ldots,\, V_k$ such that 
$E^{\otimes i} \,= \,\displaystyle \sum_{j=1}^{k}a_{i,j} V_j \ \forall \ i \geq 1$. The symmetric
product $\text{Sym}^i(E)$ is a direct summand of the vector bundle $E^{\otimes i}$. Hence from
\cite[p.~315, Theorem~3]{At} (quoted in Lemma \ref{lem0}) it follows that
$$
\text{Sym}^i(E)\,=\, \, \bigoplus_{j=1}^{k}a'_{i,j} V_j\, ,
$$
where $a'_{i,j}\,\leq\, a_{i,j}$ are nonnegative integers. For each $i\,\geq\, 1$, fix
an isomorphism of $\text{Sym}^i(E)$ with $\bigoplus_{j=1}^{k}a'_{i,j} V_j$.

Now fix smooth Hermitian structures $g_j$ on $V_j$. This produces a Hermitian structure
on any direct sum of copies of $V_j$ such that the direct sum decomposition is orthogonal.
Consequently, we get a Hermitian structure ${h}_i$ on $\text{Sym}^i(E)$ using the above
isomorphism.

Let $\Theta_j$ be the curvature of the Chern connection on the Hermitian
vector bundle $(V,\, g_j)$. So the curvature $\Theta_i$ of the Chern connection for
$(\text{Sym}^i(E),\, h_i)$ satisfies the equation
$$
\Theta_i\,=\, \bigoplus_{j=1}^{k}a'_{i,j} \Theta_i
$$
(coefficients correspond to direct sum).

According to Theorem 1.12 of \cite{DPS}, the vector bundle $E$ is nef if
for every $\epsilon \,>\,0$ and $m\,\geq \, m_0 (\epsilon)$ we have
\begin{gather}
\Theta_m \,\geq\, -m\epsilon \omega \otimes {\rm Id}_{{\rm Sym}^m(E)}
\label{DPScondition}
\end{gather}
in the sense of Griffiths. It is straight-forward to check that this condition is
satisfied. Hence the finite vector bundle $E$ is nef.
\hfill $\Box$
\end{proof}

We now prove an analogue of Theorem 1.18 of Demailly-Peternell-Schneider \cite{DPS} in our 
special case of astheno-K\"ahler manifolds.

\begin{theorem}
Suppose $(X,\,\omega)$ is a compact complex manifold equipped with a
Hermitian metric $\omega$ that is both, Gauduchon ($\pbp \omega^{n-1}\,=\,0$), as well as
astheno-K\"ahler $(\pbp \omega^{n-2}\,=\,0$). Suppose a holomorphic vector bundle $E$
of rank $r$ over $X$ satisfies $[c_2(E)]_{BC}=0$. 
Then it is numerically flat if and only if the following three conditions hold:
\begin{enumerate}
\item[\rm (1)] $E$ is $\omega$-semi-stable,

\item[\rm (2)] there is a Jordan-H\"older filtration
$$\{0\} \,=\, E_0\,\subset\, E_1\, 
\subset\, \ldots\,\subset\, E_{p-1} \,\subset\, E_p\,=\,E$$
consisting of subbundles, and

\item[\rm (3)] for all $1\, \leq\, i\, \leq\, p$, the $\omega$-stable bundle $E_i/E_{i-1}$
is Hermitian flat, i.e., given by an irreducible unitary representation of $\pi_1(X)$ in
${\rm U}({\rm rank}(E_i/E_{i-1}))$.
\end{enumerate}
\label{DPSthm}
\end{theorem}

\begin{remark}\label{vanishingfirstchernclassremark}
{\rm It can be shown that $[c_1(E)]_{BC}\,=\,[c_1(\det(E))]_{BC}\,=\,0$ whenever $E$ is numerically
flat. Indeed, applying Proposition 1.14 (ii) of \cite{DPS} to the determinant representation of
$E$ and $E^{*}$ 
we see that $\det(E)$ is numerically flat. Now Corollary 1.5 of \cite{DPS} implies that 
$\det(E)$ is Hermitian flat and hence $[c_1(\det(E))]_{BC}\,=\,0$. We are grateful to the 
anonymous referee for pointing this out.}
\end{remark}

\begin{proof}[Proof of Theorem \ref{DPSthm}]
A Hermitian flat bundle is clearly numerically flat. Now from 
\cite[Proposition~1.15(ii)]{DPS} it follows that any holomorphic vector bundle
admitting a filtration of holomorphic subbundles such that each successive quotient
admits a flat Hermitian structure is also numerically flat. Consequently, $E$ is numerically
flat if the three conditions in the theorem hold.

We prove the converse below. Assume that $E$ is numerically flat.

Our proof inducts on the rank of $E$. Note that if $S\,\subset \,E$ is a holomorphic subbundle then the vector bundle $E/S$ is nef because $E$ is
nef \cite[p.~308, Proposition~1.15(i)]{DPS}. If $S$ is nef, then applying
\cite[p.~308, Proposition~1.15(iii)]{DPS} to the exact sequence
$$
0\, \longrightarrow\, (E/S)^* \, \longrightarrow\, E^* \, \longrightarrow\, S^*
\, \longrightarrow\, 0
$$
we conclude that $(E/S)^*$ is nef. Therefore, $E/S$ is numerically flat if $S$ is so.
These observations are crucial for the induction step.

Since $\omega$ is assumed to be Gauduchon, we see that for every holomorphic line 
bundle $L$, $$\displaystyle \int c_1(L) \wedge \omega^{n-1}\,\in\,\mathbb R$$ is independent 
of the choice of Chern connection on $L$. Moreover, because $\omega$ is 
astheno-K\"ahler, for any $(2,2)$ Chern-Weil form $\Phi$, the integral
$$\int \Phi \wedge \omega^{n-2}\,\in\,\mathbb R$$
is independent of the choice of 
Chern connection, i.e., independent of the Bott-Chern representative.

Let $\mathcal{W}$ be a coherent analytic subsheaf of $\mathcal{O}(E)$ such that
the quotient $\mathcal{O}(E)/\mathcal{W}$ is torsion-free. This implies
that $\mathcal{W}$ is reflexive \cite[p.~153, Proposition 4.13]{Ko}.
The rank of $\mathcal{W}$ will be denoted by $q_w$. Then the arguments of 
Step $1$ in the proof of Theorem 1.18 in \cite{DPS} go through to show that 
$\int_X [c_1(\mathcal{W})]_{BC}\wedge \omega^{n-1}\,\leq\, 0$. Now assume that
$$\int_X [c_1(\mathcal{W})]_{BC}\wedge \omega^{n-1}\,=\, 0\, .$$ Then we have
the following:
\begin{itemize}
\item $\det(\mathcal{W})$, which is a holomorphic line bundle, is Hermitian flat,

\item $[c_1(\mathcal{W})]_{BC}\,=\,0$, and

\item the homomorphism $\det(\mathcal{W})\,\longrightarrow\, \Lambda^{q_w} E$ has no nontrivial
zeroes (that is, $\det(\mathcal{W})$ is a line subbundle of $\Lambda^{q_w} E$).
\end{itemize}
Since $\det(\mathcal{W})$
is a line subbundle of $\Lambda^{q_w} E$, it follows that $\mathcal{W}$ is a subbundle
of $E$ (recall that $\mathcal{O}(E)/\mathcal{W}$ is torsion-free); see
\cite[Lemma 1.20]{DPS}.

Now suppose $\mathcal{F} \,\subset\, \mathcal{O}(E)$ is a minimal rank $q\,>\,0$ subsheaf such that $\mathcal{O}(E)/\mathcal{F}$ is torsion-free and
\begin{equation}\label{s1}
\displaystyle \int_X [c_1(\mathcal{F})]_{BC}\wedge \omega^{n-1}\,=\,0\, .
\end{equation}
As noted before, $\mathcal{F}$ is a holomorphic subbundle of $E$ satisfying 
\begin{equation}\label{s010}
[c_1(\mathcal{F})]_{BC}=0.
\end{equation} 
We also note that
$\mathcal{F}$ is stable, because it is of minimal rank satisfying \eqref{s1}. Therefore, the Kobayashi-L\"ubke inequality says that
\begin{equation}\label{s2}
\displaystyle \int_X \left ( 2q\cdot [c_2(\mathcal{F})]_{BC}-(q-1)[c_1(\mathcal{F})^2]_{BC}\right )\wedge \omega^{n-2} \,\geq \,0\, ;
\end{equation}
see \cite[p.~107]{LYZ}. We will now show that in \eqref{s2} the equality holds.

It was noted earlier that $\mathcal{W}$ is a subbundle of $E$ if it satisfies
the condition in \eqref{s1}. Consequently, we get a filtration of holomorphic subbundles
$$
0\,=\, \mathcal{F}_0 \, \subset\, \mathcal{F}\,=:\, \mathcal{F}_1\, \subset \, \mathcal{F}_2\, \subset \, \cdots
\, \subset \, \mathcal{F}_{m-1}\, \subset \, \mathcal{F}_{m}\,=\, E
$$
such that each successive quotient $\mathcal{F}_i/\mathcal{F}_{i-1}$, $1\, \leq\, i\, \leq\, m$, satisfies
the condition
\begin{gather}
[c_1(\mathcal{F}_i/\mathcal{F}_{i-1})]_{BC}\,=\,0\, ,
\label{abovecond}
\end{gather}
and $\mathcal{F}_i/\mathcal{F}_{i-1}$ is stable. Indeed, $E/\mathcal{F}_1$ is a bundle satisfying condition (\ref{abovecond}) (by the cup product formula for the Chern classes in Bott-Chern cohomology, equation (\ref{s010}), and the fact that $[c_1(E)]_{BC}=0$ by remark \ref{vanishingfirstchernclassremark}). The arguments in step 2 of the proof of Theorem 1.18 in \cite{DPS} show that it is numerically flat. Therefore, we may induct on the rank of $E$ to get the above filtration.

\indent Let $q_i$ be the rank of $\mathcal{F}_i/\mathcal{F}_{i-1}$.
We again have the Kobayashi-L\"ubke inequality
\begin{equation}\label{s3}
\displaystyle \int_X \left (2q_i \cdot [c_2(\mathcal{F}_i/\mathcal{F}_{i-1})]_{BC}-(q_i-1)[c_1(\mathcal{F}_i/\mathcal{F}_{i-1})^2]_{BC} \right )\wedge \omega^{n-2}\,\geq \,0
\end{equation}
for every $1\, \leq\, i\, \leq\, m$. Since $[c_1(\mathcal{F}_i/\mathcal{F}_{i-1})]_{BC}=0$, we see that 
$$\displaystyle \int_X [c_2(\mathcal{F}_i/\mathcal{F}_{i-1})]_{BC}\wedge \omega^{n-2}\geq 0.$$
By the cup product formula (and $[c_1(E)^2]_{BC}=0$) we see that $$\displaystyle \int_X [c_2(E)]_{BC} \wedge \omega^{n-2} \geq 0.$$
Recall that $[c_2(E)]_{BC}\,=\,0$ by assumption; therefore, we see that equality holds in all
of the above inequalities. 

Since $\mathcal{F}$ is $\omega$-stable, by Li-Yau and L\"ubke-Teleman's generalization
(\cite{LY}, \cite{KL}) of the 
Uhlenbeck-Yau theorem \cite{UY}, it follows that $\mathcal{F}$ admits a Hermite-Einstein 
metric. Since the inequality in \eqref{s2} is an equality, the 
Hermite-Einstein metric on $\mathcal{F}$ is flat and hence has vanishing Chern classes in the Bott-Chern cohomology. Now we take the quotient $E/\mathcal{F}$ and 
use the induction hypothesis. Note that the remarks made earlier show that the quotient is numerically flat, and the cup product formula shows that it has vanishing $[c_2]_{BC}$.
\hfill $\Box$
\end{proof}

\section{Proof of Theorem \ref{mainthm}}

Any representation $W$ of a finite group $G$ has the property that there are two distinct 
polynomials $f$ and $g$ with nonnegative integral coefficients such that the two 
$G$-modules $f(W)$ and $g(W)$ are isomorphic. Indeed, this follows from the
following two facts
\begin{itemize}
\item any finite dimension $G$-module is a direct sum of irreducible
$G$-modules, and

\item there are only finitely many isomorphism classes of irreducible
$G$-modules.
\end{itemize}
Therefore, if a holomorphic vector bundle $E$ on $X$ is given by a linear
representation of $\pi_1(X)$ with finite image, then $E$ is finite.

To prove the converse, let $E$ be a finite vector bundle over a compact
complex manifold $X$ admitting a Gauduchon
astheno-K\"ahler metric. Therefore, there exists integers $m\geq 1$, $a_{ij} \geq 0$,
$r_j \geq 1$ with
$1\leq\, j\, \leq\, m$, $i\,\in\, {\mathbb N}^{>0}$,
and holomorphic vector bundles $V_j$ of ranks $r_j$ such that
\begin{equation}\label{de1}
E^{\otimes i}\,=\,\displaystyle \sum_{j=1}^{m} a_{ij} V_j\, .
\end{equation}
We may assume that the collection
of vector bundles $\{V_j\}_{j=1}^m$ satisfies the three conditions stated prior to Lemma \ref{lem0}. 

{}From Theorem \ref{vanishingofchernclassesthm}, Lemma \ref{nefness}, and Theorem \ref{DPSthm} we know that $E$ has a filtration
of holomorphic subbundles
\begin{equation}\label{l2}
\{0\} \,=\, E_0\,\subset\, E_1\,
\subset\, \ldots\,\subset\, E_{p-1} \,\subset\, E_p\,=\,E
\end{equation}
such that for all $1\,\leq\, i\, \leq\, p$, the quotient bundle $E_i/E_{i-1}$ is given by an 
irreducible unitary representation of $\pi_1(X)$.

\begin{lemma}\label{finiteHN}
For every $1\,\leq\, i\, \leq\, p$, the vector bundle $E_i/E_{i-1}$ in \eqref{l2} is finite.
\end{lemma}

\begin{proof}
From Lemma \ref{lem0}, Theorem \ref{vanishingofchernclassesthm}, Theorem \ref{DPSthm}, and
Lemma \ref{nefness} we see that each $V_j$ admits a filtration of holomorphic subbundles such that 
each successive quotient is given by an irreducible unitary representation of $\pi_1(X)$. Fix 
such a Jordan-H\"older filtration
\begin{equation}\label{de2}
\{0\} \,=\, V_{j,0}\,\subset\, V_{j,1}\,
\subset\, \ldots\,\subset\, V_{j, p_j-1} \,\subset\, V_{j,p_j}\,=\,V_j
\end{equation}
for each $1\, \leq\, j\, \leq\, m$, and consider the collection of vector bundles
\begin{equation}\label{c1}
\{V_{j,i}/V_{j,i-1}\}\, , \ \ 1\, \leq\, i\, \leq\, p_j\, ,\ 1\, \leq\, i\, \leq\, m\, .
\end{equation}
Every vector bundle in the collection in \eqref{c1} is given by an irreducible unitary
representation of $\pi_1(X)$.

We will show that every vector bundle in the collection in \eqref{c1} is finite. More
precisely, we will show that for any $V_{j,i}/V_{j,i-1}$ in \eqref{c1}, and any
integer $\alpha \,\geq\, 1$, the tensor power $(V_{j,i}/V_{j,i-1})^\alpha$ is a direct sum
of copies of vector bundles from the collection in \eqref{c1}.

To prove this, note that the filtration of $V_j$ in \eqref{de2} produces a holomorphic filtration of
subbundles
\begin{equation}\label{de4}
\{0\} \,=\, U^0_j\,\subset\, U^1_j\, \subset\, U^2_j\,\cdots\,\subset\,
V^\alpha_j
\end{equation}
of $V^\alpha_j$. The filtration $V^\alpha_j$ in \eqref{de4} has the following properties:
\begin{enumerate}
\item[\rm (1)] each successive quotient is of the form
$$
\bigotimes_{\beta=1}^{p_j} (V_{j,\beta}/V_{j,\beta-1})^{\otimes n_\beta}\, ,
$$
where $n_\beta\, \geq\, 0$ and $\sum_{\beta=1}^{p_j} n_\beta\,=\, \alpha$, and

\item[\rm (2)] $(V_{j,i}/V_{j,i-1})^\alpha$ is one of the successive quotients.
\end{enumerate}
Therefore, each successive quotient in \eqref{de4} is a direct sum of stable vector
bundles of degree zero, because it admits a flat unitary connection given by the flat
unitary connections on $V_{j,\beta}/V_{j,\beta-1}$, $1\, \leq\, \beta\, \leq\, p_j$.
It should be clarified that each successive quotient in \eqref{de4}
is polystable, but need not be stable, because the flat unitary connection on
$V_{j,\beta}/V_{j,\beta-1}$ need not be irreducible.

On the other hand, from the proof of Lemma \ref{lem0} we know that 
$$
V^\alpha_j\,=\, \sum_{\gamma=1}^{m} e_{\gamma,j} V_\gamma\, .
$$
Fix an isomorphism of $V^\alpha_j$ with the direct sum
$\sum_{\gamma=1}^{m} e_{\gamma,j} V_\gamma$. Then fix
an ordering of these direct summands of $V^\alpha_j$.
Now using the filtrations of $V_j$, $1\,\leq\, j\, \leq\, m$,
in \eqref{de2}, and the above ordering of the direct summands of $V^\alpha_j$, we get a filtration
\begin{equation}\label{de31}
\{0\} \,=\, W^0_j\,\subset\, W^1_j\,
\subset\, \cdots\,\subset\, W^{\ell_1-1}_j \,\subset\, W^{\ell_1}_j \,=\, V^\alpha_j
\end{equation}
such that each successive quotient $W^b_j/W^{b-1}_j$, $1\, \leq\, b\, \leq\, \ell_1$, is a
vector bundle from the collection in \eqref{c1}.

We recall that for a Jordan-H\"older filtration of a semi-stable holomorphic vector bundle 
$F$ with successive quotients $Q_i$ reflexive, the isomorphism class of the coherent sheaf 
$\bigoplus_i Q_i$ is unique (see
\cite[Proposition 2.1 (p.~998) and Proposition A.2 (p.~1019)]{BTT}). Consider the two
filtrations of subbundles of $V^\alpha_j$ in \eqref{de4} and \eqref{de31}. For the filtration
in \eqref{de4}, each successive quotient is a direct sum of stable vector bundles of
degree zero, while for the filtration in \eqref{de31}, each successive quotient is a
stable vector bundle of degree zero. Therefore, from Theorem 3 of \cite{At} and the above result of \cite{BTT} we conclude that
each successive quotient for the filtration
in \eqref{de4} is a direct sum of some successive quotients for the
filtration in \eqref{de31}. We noted
above that all the successive quotients $W^b_j/W^{b-1}_j$, $1\, \leq\, b\, \leq\, \ell_1$, in
\eqref{de31} are vector bundles from the collection in \eqref{c1}. Therefore, we now
conclude that each successive quotient for the filtration in \eqref{de4} is a direct sum
of copies of vector bundles from the collection in \eqref{c1}. In particular, the vector
bundle $(V_{j,i}/V_{j,i-1})^\alpha$ is a direct sum
of copies of vector bundles from the collection in \eqref{c1}. Therefore, the vector
bundle $(V_{j,i}/V_{j,i-1})^\alpha$ is finite.

Next, from \eqref{de1} we have
$$
E\,=\,\displaystyle \sum_{j=1}^{m} a_{1j} V_j\, .
$$
Fix an isomorphism of $E$ with $\displaystyle \sum_{j=1}^{m} a_{1j} V_j$. Then fix
an ordering of these direct summands of $E$.
Now using the filtrations of $V_j$, $1\,\leq\, j\, \leq\, m$,
in \eqref{de2}, and the above ordering of the direct summands of $E$, we get a filtration
\begin{equation}\label{de3o}
\{0\} \,=\, W^0\,\subset\, W^1\,
\subset\, \cdots\,\subset\, W^{\ell_0-1} \,\subset\, W^{\ell_0} \,=\, E
\end{equation}
such that each successive quotient $W^b/W^{b-1}$, $1\, \leq\, b\, \leq\, \ell_0$, is a
vector bundle from the collection in \eqref{c1}.

Consider the two filtrations of subbundles of $E$ in \eqref{l2} and \eqref{de3o}. For both of 
them, the successive quotients are stable of degree zero.
Therefore, from the above mentioned result of \cite{BTT}, namely
\cite[Proposition 2.1 (p.~998) and Proposition A.2 (p.~1019)]{BTT}, we conclude that
the graded sheaves for filtrations in \eqref{l2} and \eqref{de3o} are isomorphic. We noted
above that all the successive quotients $W^b/W^{b-1}$, $1\, \leq\, b\, \leq\, \ell_0$, in
\eqref{de3o} are vector bundles from the collection in \eqref{c1}. Therefore, using Theorem 3 of \cite{At} we now
conclude that each successive quotient $E_i/E_{i-1}$, $1\,\leq\, i\, \leq\, p$,
in \eqref{l2} is a vector bundle from the collection in \eqref{c1}.

We have already proved that every vector bundle in \eqref{c1} is finite. Hence
each successive quotient $E_i/E_{i-1}$, $1\,\leq\, i\, \leq\, p$,
in \eqref{l2} is a finite vector bundle.
\hfill $\Box$
\end{proof}

The following lemma show that $E_i/E_{i-1}$ can be trivialized by pulling back to some
finite unramified covering of $X$.

\begin{lemma}
If $X$ is a compact complex manifold admitting a Gauduchon astheno-K\"ahler
metric, and $E$ is a finite holomorphic
vector bundle on $X$, then there is a finite unramified connected Galois cover
$\pi\,:\,Y \,\longrightarrow\, X$ such that the pull back 
$\pi^* E$ admits a filtration of holomorphic subbundles with the
property that each successive quotient is a holomorphically trivial vector bundle.
\label{finiteflat}
\end{lemma}

\begin{proof}
Lemma \ref{finiteHN} shows that each stable vector bundle $E_i/E_{i-1}$ in \eqref{l2} is 
finite. Theorem \ref{DPSthm} shows that the $E_i/E_{i-1}$ has an irreducible unitary flat 
connection. We claim that the Hermite-Einstein flat connection on $E_i/E_{i-1}$ actually has 
finite monodromy. Indeed, this follows from Proposition 2.8 of \cite{scholbis}. The monodromy 
representation of the Hermite-Einstein flat connection on $E_i/E_{i-1}$ produces a finite 
unramified Galois covering of $X$.

Now take $Y$ to be a connected component of the fiber product of the $p$ unramified Galois
coverings of $X$ given by $\{E_i/E_{i-1}\}_{i=1}^p$. The pullback to $Y$ of the filtration
in \eqref{l2} has the required property.
\hfill $\Box$
\end{proof}

The following lemma shows that the holomorphic vector bundle $\pi^* E$ in Lemma 
\ref{finiteflat} is trivial.

\begin{lemma}
If $E$ is a finite holomorphic vector bundle of rank $r$ on a compact complex manifold
$Y$ such that there is a filtration of holomorphic subbundles
\begin{equation}\label{a0}
0\,=\, E_0\, \subset\, E_1\, \subset\, \cdots \, \subset\, E_{r-1} \, \subset\, E_r\,=\, E
\end{equation}
satisfying the conditions that ${\rm rank}(E_i)\,=\, i$, and
the holomorphic line bundle $E_i/E_{i-1}$ is trivial
for all $1\, \leq\, i\, \leq\, r$. Then $E$ is holomorphically trivial.
\label{extensionsbyflats}
\end{lemma}

\begin{proof}
Since $E$ is finite, there are finitely many indecomposable holomorphic vector bundles
$V_1,\, \cdots,\, V_m$ such that
\begin{equation}\label{y1}
E^{\otimes i} \,=\, \sum_{j=1}^m c_{i,j}V_j
\end{equation}
for all $i\, \geq\, 1$, where $c_{i,j}$ are nonnegative integers. By Lemma \ref{lem0} the $V_i$ can be chosen to be finite bundles themselves. Given that
$E_i/E_{i-1}$ is trivial for all $1\, \leq\, i\, \leq\, r$, each vector bundle
$E^{\otimes i}$ admits a filtration of holomorphic subbundles such that 
each successive quotient of it is the trivial line bundle. Akin to the proof of
Lemma \ref{finiteHN}, by the (weak) uniqueness of the graded object of a Jordan-H\"older
filtration in the sense of Proposition 2.1 of \cite{BTT} and from Theorem 3 of \cite{At} we see that the $V_i$ also admit such filtrations.

When $r\,=\,1$, there is nothing to prove. Next we consider the case of rank two.
Assume that
\begin{enumerate}
\item[\rm (1)] $r\,=\, \text{rank}(E)\, =\,2$, and

\item[\rm (2)] $E$ is not trivial.
\end{enumerate}
Therefore, $E$ fits in a short exact sequence of holomorphic vector bundles
\begin{equation}\label{g3}
0\,\longrightarrow\, E_1\,=\, {\mathcal O}_Y\,\longrightarrow\, E\,
\stackrel{\gamma}{\longrightarrow}\,{\mathcal O}_Y\,=\, E/E_1 \,\longrightarrow\, 0\, ,
\end{equation}
which does not split holomorphically, because $E$ is not trivial.

For any integer $i\, \geq\, 1$, let $\text{Sym}^i(E)$ denote the $i$-th symmetric 
power of $E$; it is a direct summand of $E^{\otimes i}$, meaning $\text{Sym}^i(E) 
\,\subset\, E^{\otimes i}$, and there is a holomorphic subbundle $S_i\,\subset\, 
E^{\otimes i}$ such that the natural homomorphism
$$
\text{Sym}^i(E)\oplus S_i\, \longrightarrow\, E^{\otimes i}
$$
is an isomorphism. From this it follows that there are
finitely many vector bundles $\{V'_j\}_{j=1}^{m'}$ from the collection $V_1,\ldots, V_m$ such that
\begin{equation}\label{y2}
\text{Sym}^i(E) \,=\, \sum_{j=1}^{m'} c'_{i,j}V'_j
\end{equation}
for all $i\, \geq\, 1$, where $c'_{i,j}$ are nonnegative integers. Indeed, since the $V_j$
are indecomposable, this follows from Theorem 3 of \cite{At}. Recall that each $V'_j$ 
admits a filtration of holomorphic subbundles such that each
successive quotient of it is the trivial line bundle. In particular, we have
\begin{equation}\label{y4}
\dim H^0(Y,\, V'_j)\,\geq \, 1\, , \ \ \forall \ 1\leq\, j\,\leq\, m'\, .
\end{equation}

We will now show that
\begin{equation}\label{g2}
\dim H^0(Y,\, \text{Sym}^n(E))\,=\, 1\, , \ \ \forall \ n \,\geq\, 1\, .
\end{equation}

Let $$\tau\, \in\, H^1(Y,\, {\rm Hom}({\mathcal O}_Y,\, {\mathcal O}_Y)) \,=\,
H^1(Y,\, {\mathcal O}_Y)$$
be the extension class for the short exact sequence in \eqref{g3}. We have
$$\tau\,\not=\, 0$$
because \eqref{g3} does not split holomorphically. Let
\begin{equation}\label{g4}
0\,\longrightarrow\, H^0(Y,\, {\mathcal O}_Y) \,\stackrel{\alpha}{\longrightarrow}\,
H^0(Y,\, E)
\,\longrightarrow\, H^0(Y,\, {\mathcal O}_Y) \,\stackrel{\varphi}{\longrightarrow}
\, H^1(Y,\, {\mathcal O}_Y)
\end{equation}
be the long exact sequence of cohomologies for the short exact sequence in \eqref{g3}.
Since the homomorphism $\varphi$ in \eqref{g4} sends $1\,\in\, H^0(Y,\, {\mathcal O}_Y)\,=\,
\mathbb C$ to $\tau\, \in\, H^1(Y,\, {\mathcal O}_Y)$ (this follows from the construction
of the extension class), we conclude that $\alpha$ in \eqref{g4} is
an isomorphism. Therefore, \eqref{g2} holds for $m\,=\,1$.

Now we will prove \eqref{g2} using induction on $m$.

Suppose \eqref{g2} holds for all $n\, \leq\, n_0-1$, with $n_0\, \geq\, 2$. To
prove \eqref{g2} for $n\,=\, n_0$, consider
the short exact sequence of holomorphic vector bundles
$$
0\,\longrightarrow\, \text{Sym}^{n_0-1}(E)\,\longrightarrow\, \text{Sym}^{n_0}(E)
\,\stackrel{\gamma^{\otimes n_0}}{\longrightarrow}\,
\text{Sym}^{n_0}({\mathcal O}_Y)\,=\, {\mathcal O}_Y
\,\longrightarrow\, 0\, ,
$$
where $\gamma$ is the projection in \eqref{g3}.
The above inclusion of $\text{Sym}^{n_0-1}(E)$ in $\text{Sym}^{n_0}(E)$ is given by the
natural inclusion $E^{\otimes (n_0-1)}\otimes E_1 \, \hookrightarrow\,
E^{\otimes (n_0-1)}\otimes E\,=\, E^{\otimes n_0}$ (see \eqref{g3}). Let
\begin{equation}\label{g5}
0\,\longrightarrow\, H^0(Y,\, \text{Sym}^{n_0-1}(E)) \,\stackrel{\beta}{\longrightarrow}
\, H^0(Y,\, \text{Sym}^{n_0}(E)) \,\longrightarrow\, H^0(Y,\, {\mathcal O}_Y) \,
\stackrel{\varphi'}{\longrightarrow} \, H^1(Y,\, \text{Sym}^{n_0-1}(E))
\end{equation}
be the long exact sequence of cohomologies associated to the above short
exact sequence of vector bundles. Now consider the composition
$$
(\gamma^{\otimes (n_0-1)})_*\circ \varphi'\,:\, H^0(Y,\, {\mathcal O}_Y)
\,\longrightarrow\, H^1(Y,\, \text{Sym}^{n_0-1}({\mathcal O}_Y))\,=\,
H^1(Y,\, {\mathcal O}_Y)\, ,
$$
where $\varphi'$ is the homomorphism in \eqref{g5}, and $$\gamma^{\otimes (n_0-1)}\,:\,
\text{Sym}^{n_0-1}(U)\,\longrightarrow\, \text{Sym}^{n_0-1}({\mathcal O}_Y)\,=\,
{\mathcal O}_Y$$
is the homomorphism given by $\gamma$ in \eqref{g3}. This
$(\gamma^{\otimes (n_0-1)})_*\circ \varphi'$ sends $1\,\in\, H^0(Y,\, {\mathcal O}_Y)\,=\,
\mathbb C$ to $(n_0-1)\tau\, \in\, H^1(Y,\, {\mathcal O}_Y)$, which is nonzero because
$\tau\,\not=\, 0$. So, the homomorphism $(\gamma^{\otimes (n_0-1)})_*\circ \varphi'$ is injective,
implying that $\varphi'$ is injective. Consequently, the homomorphism $\beta$ in \eqref{g5}
is an isomorphism. This completes the proof of \eqref{g2}.

{}From \eqref{y2}, \eqref{g2} and \eqref{y4} it follows that
$\sum_{j=1}^{m'} c'_{i,j}\, \leq\, 1$ for all $i\, \geq\, 0$. This implies that
$$
i+1\,=\, \text{rank}(\text{Sym}^i(E))\, \leq\,
\max_{1\leq j\leq m'}\ \text{rank}(V'_j) \, , \ \ \forall \ i \,\geq\, 1\, .
$$
In view of this contradiction we conclude that the lemma holds for $r\,=\,2$.

We will now prove the general case.

Assume that
\begin{enumerate}
\item[\rm (1)] $r\,=\, \text{rank}(E)\, >\,2$, and

\item[\rm (2)] $E$ is not trivial.
\end{enumerate}

We will show that there is a subbundle $F\, \subset\, E$
with $2\, \leq\, \text{rank}(F)\,=\, s\, \leq\, r\,=\, \text{rank}(E)$, and a filtration of
holomorphic subbundles
\begin{equation}\label{g0}
0\,=\, F_0\, \subset\, F_1\, \subset\, \cdots \, \subset\, F_{s-1} \, \subset\, F_s\,=\,
F\, ,
\end{equation}
satisfying the following conditions:
\begin{enumerate}
\item[\rm (1)] $\text{rank}(F_i)\,=\, i$ for all $1\, \leq\, i\, \leq\, s$,

\item[\rm (2)] the holomorphic line bundles $F_i/F_{i-1}$, $1\, \leq\, i\, \leq\, s$, are
trivial, and

\item[\rm (3)] the rank two holomorphic vector bundle $F/F_{s-2}$ is not trivial.
\end{enumerate}
To construct a filtration as in \eqref{g0}, take $s$ to be the smallest positive integer
such that the vector bundle $E_s$ in \eqref{a0} is nontrivial. So $s\, \geq\, 2$,
and the vector bundle $E_{s-1}$ is isomorphic to ${\mathcal O}^{\oplus (s-1)}_Y$.
Moreover, $E_s$ fits in a short exact sequence of holomorphic vector bundles
$$
0\,\longrightarrow\, {\mathcal O}^{\oplus (s-1)}_Y\,\longrightarrow\, E_s
\,\longrightarrow\, {\mathcal O}_Y\,=\, E_s/E_{s-1} \,\longrightarrow\, 0\, .
$$
The corresponding extension class $$\delta\, :=\, (\delta_1,\, \cdots,\,
\delta_{s-1})\, \in\, H^1(Y,\, \text{Hom}({\mathcal O}_Y,\, {\mathcal O}^{\oplus (s-1)}_Y))
\,=\, H^1(Y,\, {\mathcal O}_Y)^{\oplus (s-1)}
$$
is nonzero, where $\delta_i\, \in\, H^1(Y,\, {\mathcal O}_Y)$,
because $E_s$ is nontrivial, implying that the above short exact sequence of vector
bundles does not split holomorphically. Choose an integer $j_0\, \in\, [1,\, s-1]$ such that
$\delta_{j_0}\, \not=\, 0$. Now in \eqref{g0}, take
\begin{itemize}
\item $F_i\,=\, E_i$ for all $0\, \leq\, i\, \leq\, j_0-1$ (see \eqref{a0}),

\item $F_i\,=\, E_{i+1}$ for all $j_0\, \leq\, i\, \leq\, s-2$,

\item $F_{s-1}\, =\, E_{j_0}$, and

\item $F\,=\, F_s\,=\, E_s$.
\end{itemize}
This filtration clearly satisfies all the three conditions.

Note that for all $i\, \geq\, 1$, the symmetric product $\text{Sym}^i(F)$ is a
subbundle of $\text{Sym}^i(E)$, and $\text{Sym}^i(F/F_{s-2})$ is a quotient of
$\text{Sym}^i(F)$. Since $F/F_{s-2}$ is a nontrivial extension of ${\mathcal O}_Y$ by
${\mathcal O}_Y$, we have
\begin{equation}\label{y5}
\dim H^0(Y,\, \text{Sym}^i(F/F_{s-2}))\,=\, 1\, , \ \ \forall \ i \,\geq\, 1\, ,
\end{equation}
as proved in \eqref{g2}; note that the proof of \eqref{g2} only uses that the
extension in \eqref{g3} is nontrivial, and does not use that the rank two
vector bundle is finite.

Since $\text{Sym}^i(F/F_{s-2})$ admit a filtration of holomorphic subbundles such that each 
successive quotient of it is the trivial line bundle, from \eqref{y5} it follows that 
$\text{Sym}^i(F/F_{s-2})$ is indecomposable for all $i\,\geq\, 1$. Consequently, the 
indecomposable components (direct summands) of $\text{Sym}^i(E)$ are of arbitrarily large rank 
as $i$ goes to infinity. Since $\text{Sym}^i(E)$ is a direct summand of $E^{\otimes i}$, the 
indecomposable components of $E^{\otimes i}$ are also of arbitrarily large rank as $i$ goes to 
infinity. But from \eqref{y1} it follows that these ranks are bounded by $\max_{1\leq j\leq m}\ 
\text{rank}(V_j)$.

In view of the above contradiction, the lemma is proved.
\hfill $\Box$
\end{proof}

The vector bundle $\pi^*E$ in Lemma \ref{finiteflat} is finite because $E$ is finite. 
So from Lemma \ref{extensionsbyflats} we know that $\pi^*E$ is holomorphically 
trivial. Let $D$ be a connection on $\pi^*E$ given by a trivialization of it (this 
connection does not depend on the choice of the trivialization). Let 
$\Gamma\,:=\,\text{Gal}(\pi)$ be the Galois group for the Galois covering $\pi$. For 
any $g\, \in\, \Gamma$, the vector bundle $g^*\pi^*E$ is trivial, and any isomorphism 
between $\pi^*E$ and $g^*\pi^*E$ takes the connection $D$ on $\pi^*E$ to the 
connection $g^*D$ on $g^*\pi^*E$. Therefore, $D$ descends to a connection on $E$; 
this descended connection will be denoted by $D_E$. We note that $D_E$ is flat 
because $D$ is so. The monodromy representation for $D_E$ clearly factors through the 
quotient $\pi_1(X)\, \longrightarrow\, \Gamma$. Therefore, $E$ admits a flat 
connection with finite monodromy group. This completes the proof of Theorem 
\ref{mainthm}.

\paragraph*{Acknowledgements.}
The authors are sincerely grateful to the anonymous referee for pointing out subtle errors in our 
earlier proofs and giving other useful feedback. The work of the second author (Pingali) was 
partially supported by SERB grant No. ECR/2016/001356. The second author also thanks the Infosys 
foundation for supporting him through the Infosys young investigator grant. The first author is 
supported by a J. C. Bose Fellowship.

\providecommand{\bysame}{\leavevmode\hbox to3em{\hrulefill}\thinspace}
%
%

\bibliographystyle{amsalpha}
\bibliographymark{References}
\def\cprime{$'$}

\end{document}